\documentclass{amsart}
\usepackage{latexsym}
\usepackage{amsmath,amssymb,amsfonts,amsthm,eufrak}
\usepackage{mathrsfs}
\usepackage{upref}
\usepackage{amscd}
\usepackage{eucal}
\usepackage[all]{xy}
\usepackage[dvips]{graphicx}

\title{Toward a conjecture of Tan and Tu on fibered general type surfaces}
\author{A. Huitrado-Mora, M. Casta\~neda-Salazar and A. G. Zamora}
\address{Centro de Ciencias Matem\'aticas,\newline Universidad Nacional Aut\'onoma de M\'exico, Campus
Morelia.\newline  Apartado Postal 61-3, Santa Mar\'{\i}a, 58089.
\newline Morelia, Michoac\'an, M\'exico.} \email{huitrado at matmor.unam.mx}
\address{Centro de Ciencias
Matem\'aticas,\newline Universidad Nacional Aut\'onoma de
M\'exico, Campus Morelia.\newline  Apartado Postal 61-3, Santa
Mar\'{\i}a, 58089.\newline Morelia, Michoac\'an, M\'exico.}
\email{mcastaneda at matmor.unam.mx}
\address{U. A. Matem\'aticas, U.de
Zacatecas\newline Camino a la Bufa y Calzada Solidaridad, C.P.
98000\newline Zacatecas, Zac. M\'exico} \email {alexiszamora06 at
gmail.com}
\date{03/25/16}
\thanks{This paper was partially supported by CONACyT Grant 257079. First and second authors were supported by Conacyt
Doctoral Scholarships. Third author was partially supported by
Conacyt Grant 265621 for Academic Sabbatical Year}

\begin{document}

\newtheorem{teo}{Theorem}[section]
\newtheorem{lem}[teo]{Lemma}
\newtheorem{defin}[teo]{Definition}
\newtheorem{prop}[teo]{Proposition}
\newtheorem{cor}[teo]{Corollary}
\newtheorem{remark}[teo]{Remark}

\begin{abstract} Given a semistable non-isotrivial fibered surface $f:X\to \mathbb{P}^1$
it was conjectured by Tan and Tu that if $X$ is of general type,
then $f$ admits at leats $7$ singular fibers. In this paper we
prove this conjecture in several particular cases, i.e. assuming
$f$ is obtained from blowing-up the base locus of a transversal
pencil on an exceptional minimal surface $S$ or assuming that $f$
is obtained as the blow-up of the base locus of a transversal and
adjoint pencil on a minimal surface.
\end{abstract}

\maketitle

\section{Notation}

We work on the complex field number $\mathbb{C}$. All the
considered varieties will be assumed irreducible and projective.
Through the paper we shall use the following notation:

\begin{itemize}

\item[.] $X$ will be a general type surface and $S$ its minimal
model. $\pi: X \to S$ will be the associated chain of
blowing-downs.

\item[.] $f: X \to \mathbb{P}^1$ will be a semi-stable,
non-isotrivial fibration, and $F$ the general fibre. We set $g=$
the genus of $F$. By $C$ we will denote the image of $F$ under
$\pi$ and

$$\Lambda: S \dashrightarrow \mathbb{P}^1,$$ the pencil induced
by $f$. We denote by $s$ the number of singular fibers.

\item[.] We shall say that $\Lambda$ is transversal if its general
member $C\in \Lambda$ is non-singular and intersects transversally
any other general member $C'\in \Lambda$.

\item[.] We'll freely use the standard notation in surfaces'
theory. In particular $q=h^1(X,\mathcal{O}_X)$ will be the
irregularity of $X$ and $p_g= h^0(X,K_X)$ its geometric genus. By
$e(X)$ will be denote the topological Euler characteristic.

\item[.] Given divisors $D_1$ and $D_2$ in an algebraic surface we
denote as usual $D_1\equiv D_2$ for the numeric equivalence and
$D_1\sim D_2$ for the linear one. Most of the time we will be
working on regular surfaces and in this case we indistinctly use
both symbols.

\item[.] The number $m$ will be:

$$m:= K_S^2- K_X^2= e(X)-e(S).$$

Note that in general $m\le C^2$. Adjunction Formula gives the
important inequality $C.K_S+m\le 2(g-1)$ with equality holding if
$\Lambda$ is transversal.

\end{itemize}

\section{Introduction}

Let $f:X\to \mathbb{P}^1$ be a non-isotrivial semistable fibered
surface. It is a classical result that such a fibration  admits a
certain number of singular fibers (in contrast to the case when
the base of the fibration is not rational or elliptic \cite{bpv}).
In the seminal paper \cite{beauv} it was proved that this number
$s$ must be at least $4$. Subsequently the bound have been
sharpened to $s\ge 5$ if $g\ge 2$ and $s\ge 6$ if the surface $X$
is not birationally ruled (\cite{liu} \cite{tan}, \cite{ttz}). It
was conjectured by Tan and Tu, in a preprint previous to
\cite{ttz} that this bound must raise to $7$ if $X$ is of general
type (Tan-Tu conjecture for what follows). They also proved the
conjecture for genus $2\le g \le 4$ and characterized fibrations
of genus $5$ with $s=6$ on a general type surface as those
obtained from blowing up the base locus of a transversal pencil on
a Horikawa surface. Using this characterization the proof of the
conjecture for $g=5$ was completed in \cite{zam}.

Roughly speaking the proof of these bounds are based, in case
$g\ge 2$ on the canonical class inequality:

$$ K_f^2 < 2(g-1)(2g_B-2+ s),$$
for any non iso-trivial semi-stable fibration $f:X \to B$ of genus
$g\ge 2$. Here $K_f$ is the relatively canonical divisor,
$K_f:=K_X-f^*K_B$, which turns out to be $K_X(-2F)$ if
$B=\mathbb{P}^1$. The bounds for $s$ are obtained from the
positivity  properties of $K_f$ and $K_f(-F)$.

Unfortunately, this approach is useless for proving Tan-Tu
conjecture, since in this case the only relevant information the
inequality provides is that $K_X^2< 0$ if $s=6$. However for most
of the cases a fibered surface (of general type or not) satisfies
$K_X^2< 0$. Indeed, such a surface is obtained by blowing up the
base locus of some pencil $\Lambda$ on a minimal surface $S$.

In this paper we deal with Tan-Tu conjecture in some particular
cases. We first impose on the minimal model $S$ of $X$ the
conditions of being exceptional in the sense that either
$K_S^2=2$, $p_g=3$ or $K_S^2=1$ and $p_g=2$ (see \cite{bpv},
VII.8) and assuming that $f$ is obtained as the blowing-up of the
base locus of a transversal pencil $\Lambda$ in $S$. In these
cases we are able to prove the conjecture using explicit
descriptions of these surfaces as double coverings of rational
surfaces, by means of the canonical or bi-canonical map. This is
the content of Section 3, Theorems \ref{K^2=2p=3} and
\ref{K^2=1p=2}. In this sense there is some hope of extending the
result to a wider class of surfaces.

Next, in section 5 we prove the conjecture assuming that the
pencil $\Lambda$ is adjoint, i.e. $C=B+K_S$ with $B$ a big and nef
divisor in $S$ and $K_S^2\ge 3$ and for $K_S^2\le 2$ assuming not
only that $\Lambda$ is adjoint, but also transversal (Theorem
\ref{Adjoint}).The case $K_S^2=p_g=1$ is the subtler and is stated
and proved in Proposition \ref{K^2=1p=1}.

We list below the cases in which Tan-Tu conjecture have been
proved in this article:

\begin{itemize}

\item[.] If $\chi(\mathcal{O}_S)=1$ with no extra assumption on $S$
or $\Lambda$.

\item[.] If either $K_S^2=1$, $p_g=2$ or $K_S^2=2$, $p_g=3$
assuming $\Lambda$ is transversal.

\item[.] If $\Lambda$ is  transversal and adjoint and $K_S^2\le 2$ or if $\Lambda$ is merely adjoint and $K_S^2\ge 3$.

\end{itemize}

\section{Some general facts and results}

The following inequality will be systematically used: given a
semi-stable non-isotrivial  fibration of genus $g\ge 2$, $f:X\to
B$, for any integer $e\ge 2$:

$$\frac{1}{3}e^2(K_X^2-2 (g-1)(6(g_B-1)+s-s/e))\le e_f.$$

Original formulation involves the number of $(-2)$ vertical curves
in $X$, but for our purposes this version will be sufficient. The
proof can be founded in \cite{tan} and is based on successive
changes of the base $B$ of the fibration. We call this Tan's
inequality. In particular if $B= \mathbb{P}^1$ (our interest's
case) and $s=6$ we obtain:

\begin{equation} \label{TI} \frac{1}{3}e(K_X^2e+12(g-1))\leq
e_f.\end{equation}

Useful forms of this inequality are collected in the following:

\begin{lem}\label{TanIneq} Let $f:X\to \mathbb{P}^1$ be semistable, non-isotrivial of genus $g\ge 2$. If $s=6$, then evaluating (\ref{TI}) we obtain:

\begin{itemize}

\item[i)] $$K_S^2+C.K_S\le 3\chi (\mathcal{O}_S) \text{ if }
e=3,$$

\item[ii)] $$19K_S^2+18C.K_S \le m +36\chi(\mathcal{O}_S) \text{
if } e=4,$$

\item[iii)] $$7K_S^2+6C.K_S\le m + 9\chi(\mathcal{O}_S) \text{
if } e=5.$$

\end{itemize}

\end{lem}

\begin{proof} Evaluate (\ref{TI}) at the indicated value of $e$
and substitute:

$$m+C.K_S\le 2(g-1),$$

$$e(X)= 12\chi(\mathcal{O}_X)-K_X^2 \text{ (Noether's Formula)},$$

$$K_X^2=K_S^2-m,$$
and

$$e_f= 4(g-1) + e(X) \text{ (because $f$ is semistable) }.$$
\end{proof}

We start by sharpening the bound for $m$ obtained in \cite{ttz},
inequality (3.2) (compare with the proof of Theorem 2.1(4) in
\cite{ttz}).

\begin{lem}\label{Cotam} Assume $f:X\to \mathbb{P}^1$ is a semistable fibration and the minimal model $S$ of $X$ is a general
type surface, then:

 $$m\leq C^2\leq \frac{4(g-1)+K_S^2-\sqrt{8(g-1)K_S^2+ (K_S^2)^2}}{2}.$$
\end{lem}

\begin{proof}
From Index Hodge Theorem applied to $K_S$ and $C$ we get:

$$ mK_S^2\le C^2K_S^2\leq (C.K_S)^2.$$

Adjoint formula  gives $C.K_S\le 2(g-1)-m,$ therefore
$$0\le m^2-(4(g-1)+K_S^2)m+4(g-1)^2. \label{I1}$$

Consider the right hand term of the previous inequality as a
polynomial in $m$. Its discriminant turns out to be

$$\Delta=8(g-1)K_S^2 + (K_S^2)^2.$$

Thus its roots are:
$$m_{\pm }=\frac{4(g-1)+K_S^2\pm\sqrt{\Delta}}{2}.$$
It follows that either $m\le m_{-}$ or $m\ge m_{+}$. We claim that
$m\le m_{-}$ is the only possible case. Indeed, if

$$ \frac{4(g-1)+K_S^2+\sqrt{\Delta}}{2} \le m,$$
then $C.K_S\le 2(g-1)-m\leq 0$ that is impossible.
\end{proof}

As a consequence of Lemma \ref{Cotam} we obtain our first general
fact concerning the number $s$:

\begin{prop}\label{CotaGenus}
Let $S$ be of general type. If $s=6$, then
$$K_S^2+\sqrt{8(g-1)K_S^2+(K_S^2)^2}\leq 6\chi(\mathcal{O}_S).$$
In particular, if $g\geq 6$ and $\chi(\mathcal{O}_S)=1$, then
$s\geq 7.$
\end{prop}

\begin{proof}
Assume $s=6$, by Tan's inequality:

$$\frac{1}{3}e(K_X^2e+12(g-1))\leq e_f,$$
for any natural number $e\ge 2$. Evaluating in $e=3$  we obtain:
\begin{equation}\label{ec1}
3K_X^2+12(g-1)\leq e_f.
\end{equation}

Since $f$ is semistable $e(X)=-4(g-1)+e_f$ and (\ref{ec1})
becomes:

\begin{equation}\label{ec2}
3K_X^2+12(g-1)\leq e_f=e(X)+4(g-1).
\end{equation}
By definition $m=K_S^2- K_X^2=e(X)-e(S)$. This, combined with
Noether Formula $e(S)=12\chi(\mathcal{O}_S)-K_S^2$ leads to:

$$K_S^2+ 2(g-1)\leq 3\chi(\mathcal{O}_S) + m.$$

The desired inequality follows after applying Lemma \ref{Cotam}.
\end{proof}

The importance of this Proposition is that given a family of
general type surfaces with given invariants $K_S^2$ and
$\chi(\mathcal{O}_S)$ only a finite numbers of values of $g$ must
be discharged in order to conclude that a fibered surface
birational to $S$ has at least $7$ singular fibers. This principle
will be illustrated in the next section.

Note that in Proposition \ref{CotaGenus} the hypothesis of being
$S$ of general type is essential (\cite{beauv}, Example 2).

\section{Fibrations obtained from exceptional surfaces}

Surfaces satisfying either $K_S^2=2$ and $p_g=3$ or $K_S^2=1$ and
$p_g=2$ are called exceptional because of the behavior of the
tri-canonical map (Theorem VII 8.3 in \cite{bpv}). In this section
we study fibrations in these surfaces by means of the canonical
and bi-canonical map, respectively.

\begin{teo}\label{K^2=2p=3} Assume $S$ satisfies $K_S^2=2$ and
$p_g=3$. Let $f:X\to \mathbb{P}^1$ be obtained as the blowing-up
of a transversal pencil $\Lambda$ in $S$. Then, $s\ge 7$.\end{teo}

\begin{proof} By Debarre's Inequality (\cite{deb}) we know that $q=0$ and
therefore $\chi(\mathcal{O}_S)=4$. Assume $s=6$, by Proposition
\ref{CotaGenus} it is sufficient to consider:

$$6\le g \le 31.$$

In this case the canonical map $\phi_{K_S}$ defines a $2:1$
covering:

$$\phi_{K_S} : S \to \mathbb{P}^2,$$ ramified along $R\equiv
4K_S$ (\cite{hor1}). Consider the restriction $\phi :=
\phi_{K_S}\vert_C.$

First, we consider the case $\phi$ is a $2:1$ covering. Denote by
$G\subset \mathbb{P}^2$ the image of $C$ under $\phi$ and $d$ for
its degree. We have $C=\phi^*G$ and therefore:

$$m=C^2=(\phi^*G)^2=2G^2=2d^2.$$

On the other hand, taking into account that $\phi^* H=K_S$, with
$H$ a hyperplane section (i.e. the divisor associated with
$\mathcal{O}_{\mathbb{P}^2}(1)$) we obtain:

$$d= \frac{C.K_S}{2} \text{ and } m=\frac{(C.K_S)^2}{2}.$$

From this we get $2m=(2(g-1)-m)^2$. The possible values of $g$
satisfying such a relation in the range $6\le g \le 31$ are: $g=7,
13, 21, 31$. The corresponding values of $m$ and $C.K_S$ are
listed below:

\begin{center} \begin{tabular} {|c|c|c|} \hline $g-1$ & $C.K_S$ & $m$ \\ \hline 6 & 4 & 8 \\ \hline 12 & 6 & 18 \\ \hline 20
& 8 & 32 \\ \hline 30 & 10 & 50 \\ \hline \end{tabular}
\end{center}

If $g-1=20$ or $30$ we use the inequalities of Lemma \ref{TanIneq}
ii), in order to obtain a contradiction.

Assume $g-1=6$. In this case the number of singular points in the
fibers of $f$ will be:

$$\begin{aligned} e_f &= e(X)+4(g-1)\\ &=-K_X^2+ 12\chi(\mathcal{O}_X)+ 4(g-1) \\
&=6+48+24=78.\end{aligned}$$

Since $s=6$ there exists at least a singular fiber $F_0$ of $f$
containing $\sigma_0= 13=78/6$ singular points. But then, denoting
$F_0=F_1+...+F_l$ for the decomposition into irreducible
components:

$$\begin{aligned} 6=(g-1)&= \sum_{i=1}^l(g_i-1)+ \sigma_0 \\ & \ge \sum (g_i-1)+ 13,\end{aligned}$$
with $g_i$ standing for the geometric genus of $F_i$ and
$\sigma_0$ for the number of singular points of $F_0$. We have:

$$\sum_{i=1}^l (g_i-1)\le -7.$$

In particular, $F_0$ have at least $7$ irreducible rational
components. Being $\phi$ a $2:1$ covering and $C$ be applied under
$\phi$ to a plane curve $G$ of degree $d=C.K_S/2=2$ we have that
the number $l$ of irreducible components of $F_0$ is at most $4$.
In this way we get a contradiction with the assumption $s=6$.

The case $g-1=12$ follows after similar considerations, this time
tanking into account that the covering $\phi$ sends $C$ onto a
curve $G$ of degree $3$.

Consider now the case $\phi$ restricted to $C$ is $1:1$. In this
case there exists a curve $C'$ such that

$$C+C'=\phi^*G.$$

Moreover, since the ramifications of $\phi$ over $C$ occurs
exactly on the intersections of $C$ and $C'$ we have
$C.C'=C.R=C.4K_S$. Similarly, we conclude that $C.K_S=C'.K_S$.
Also, we have that:

$$2C.K_S=\phi^* G.K_S=2d,$$ i.e. $d=C.K_S$.

It follows that $\phi^*G=\phi^*((C.K_S)H)=(C.K_S)K_S$. From

$$C+C'=(C.K_S)K_S,$$ we obtain, after intersecting with $C$:

$$4K_S.C= C.C'=(C.K_S)^2-m.$$

The possibilities for such a relationship are listed below:

\begin{center} \begin{tabular} {|c|c|c|}\hline $(g-1)$ & $C.K_S$ & $m$ \\ \hline 5 & 5 & 5 \\ \hline 9 & 6 & 12
\\ \hline 14 & 7 & 21 \\ \hline 20 & 8 & 32 \\ \hline 27 & 9 & 45
\\ \hline
\end{tabular} \end{center}

Using Lemma \ref{TanIneq} we obtain a contradiction in the
following cases: $g-1=20$ and $27$ when evaluating at $e=4$, and
$g-1=5, 9$ when evaluating at $e=5$.

Finally, we must analyze the case $g-1=14$. In this case,

$$\begin{aligned} e_f &= 4(g-1)-K_X^2+12\chi(\mathcal{O}_X) \\ &
=123. \end{aligned}$$

Assuming $s=6$ there must to exist a singular fiber $F_0$ with its
number of nodes $\sigma_0\ge 21$. Assume $F_0=F_1+...+ F_l$ is its
decomposition into irreducible components. Denoting by $g_i$ the
geometric genus of $F_i$ we have:

$$14=\sum_{i=1}^l (g_i-1)+\sigma_0 \ge \sum_{i=1}^l(g_i-1) +21.$$

From this we get $l\ge 7$. Now, the image of $F_0$ in
$\mathbb{P}^2$ (under the composition $\phi\circ \pi$) is a degree
$7$ curve $G_0$. It follows that $l=7$, $G_0=L_1+..+L_7$ is the
sum of seven lines $L_i$,  and $C_0:= \pi(F_0)$ must be the sum of
seven irreducible components $C_0=C_1+...+ C_7$ ( $C_i=\pi(F_i)$).
We have, moreover, $\phi^*L_i=C_i+ C'_i\equiv K_S$,
$C_i^2={C'_i}^2=-3$ and $C_iC'_i=4$.

A simple cohomological computation shows that
$h^0(\mathcal{O}_S(C_{i_1}+C_{i_2}+C_{i_3})=1$ for any indexes
$i_1,i_2,i_3\in \{1,...,7\}$ and
$h^0(\mathcal{O}_S(C_1+...+C_4))=2$. This, together with
$(C_1+...+C_4)^2=0$ means that $\vert C_1+...+C_4\vert$ is a base
point free pencil. Call $\Delta:= C_1+...+C_4$. We have,
$C'_i.\Delta=0$ for $i=5,6,7$, thus $C'_5+C'_6+C'_7$, being
connected, must be a vertical divisor with respect to $\vert
\Delta \vert$. We conclude the existence of an effective divisor
$D'$ such that $\Delta \sim C'_5+C'_6+C'_7+D'$. It is easy to
deduce that $D'$ is a rational $(-3)-$curve and $\phi(D')$ is a
line in $\mathbb{P}^2$.

From this relation and

$$(C_1+C'_1)+...+(C_7+C'_7)\sim 7K_S,$$
it follows that:

$$C\sim C_1+...+C_7 \sim 3K_S+ D'.$$

This gives a contradiction, as the image $\phi(C)$ is a degree $7$
curve and the image $\phi(3K_S+D')$ is a degree $4$ curve.

\end{proof}

\begin{teo}\label{K^2=1p=2}
\label{teo3} Assume $f:X\rightarrow\mathbb{P}^1$ is obtained as a
blow up of a transversal pencil $\Lambda$ on a minimal surface $S$
with  $K_S^2=1$ and $p_g=2$. Then $s\geq7$.
\end{teo}

\begin{proof}
By Proposition~\ref{CotaGenus} we must consider only the values
$6\leq g\leq37.$

Considering the classical Horikawa's construction and notation
(\cite{hor}), let $\Pi: \bar S \to S$ be the blowing up centered
in the base point $p$ of $|K_S|$, denote by $E$ its exceptional
divisor and consider the ramified double covering:

$$\phi_2:\bar S \rightarrow\mathbb{F}_2.$$

The map $\phi_2$ is given as follow: the bicanonical map of $S$
determines a double cover on the singular quadric $Q\subset
\mathbb{P}^3$, the singular point being the image of $p$.
$\phi_{2}$ is the induced map on $\bar S$ after considering the
desingularization $\mathbb{F}_2$ of the quadric. The locus branch
of $\phi_2$ is the divisor $B=6\Delta+10\Gamma$, with $\Delta$ and
$\Gamma$ denoting respectively the class of the (-2)-section and
the class of the fiber in $\mathbb{F}_2$ of the structural
morphism and the ramification divisor $R$ is $5K_S+E$. Here and in
what follows given $D$ any divisor in $S$ we just write $D$ for
the divisor  $\Pi^* D$ in $\bar S$.

Denote by $\bar \Lambda$ the induced pencil on $\bar S$. Depending
on whether $p$ is a base point of $\Lambda$ or not we have

\begin{equation}
\label{Cbar} \bar C =\left\{
\begin{array}{lrr}
\Pi^*C & \textrm{if} & p\not\in\Lambda \\
\Pi^*C-E & \textrm{if} & p\in\Lambda .
\end{array}
\right.
\end{equation}
Let $G$ be the image of $\bar {C}$ under $\phi_2$. If we denote
$G=a\Delta+b\Gamma$ and considering that $\phi^*\Delta=2E$,
$\phi^*\Gamma= K_S-E$ we have

\begin{equation}
\label{G*} \phi^*G=b K_S+(2a-b)E.
\end{equation}

Let be $\phi:=\phi_2\vert_{\bar C}:\bar C\rightarrow G$ with $\deg
\phi=n=1$ or $2$.

We analyze the two cases $n=1$ or $2$ and within each one the
subcases $p\in \Lambda$ or not.

{\emph{Case 1}} $\phi$ is 2:1. In this case $\phi^*G=\bar C$.

Assume first that $p\not\in\Lambda$.

By (\ref{Cbar}) and (\ref{G*}) we have that $2a-b=0$, using this
$\bar C^2=m=b^2$. On the other hand $m=bC.K_S$, therefore
$b=C.K_S$. The next table shows the possible values of $m$,
$C.K_S$ and $g-1$.

\begin{center}
\begin{tabular}{|c|c|c|c|}\hline
$C.K_S$  & $a$ & $m$ & $g-1$ \\
\hline 4 & 2 & 16 & 10 \\ \hline 6 & 3 & 36 & 21 \\ \hline 8 & 4 &
64 & 36 \\ \hline
\end{tabular}
\end{center}

Assume $g-1=10$. In this case the number of singular points in the
fibers of $f$ will be:

\begin{eqnarray}
e_f & =12\chi(\mathcal{O}_X)-K_X^2+4(g-1) \nonumber \\
& = 36+15+4(10)=91 \nonumber
\end{eqnarray}

Since $s=6$ there exists at least a singular fiber $F_0$ of $f$
containing 16 singular points. Let $F_0=F_1+\ldots+F_l$ be the
decomposition into irreducible components, then:

\begin{eqnarray}
10=g-1 & \geq\sum_{i=1}^l(g_i-1)+\sigma_0 \nonumber \\
 & \geq\sum(g_i-1)+16
\end{eqnarray}
where $g_i$ denotes the geometric genus of $F_i$. This imply that
$l\geq6$. Moreover there are at least 6 of this components that
are rational curves and are mapping onto rational components of
$G_0$.

From \cite{hart} (Corollary V.5.18) we know that the possible
irreducible curves in $\mathbb{F}_2$ are: $\Gamma$, $\Delta$ and
$\alpha\Delta+\beta\Gamma$ with $\alpha>0$, $\beta\geq2a$. Let
$G_0=G_1+\ldots+G_s$ be the decomposition into irreducible
components.

Note that even being $\Delta$ a rational curve, is not a
possibility for any of the $G_i$'s, that because if $G_1=\Delta$
then
$$\bar C_1=\phi^*G_0=2E$$
that contradicts the semistability of $f$.

With respect to the components $G_i=a_i\Delta +b_i\Gamma$,
$a_i>0$, $b_i\geq 2a$,  $G_i$ is rational if and only if $a_i=1$
and $b_i\geq2$. Since $a=\sum_{i=1}^sa_i$ and $b=\sum_{i=1}^sb_i$
the only possible decomposition is $G_0=G_1+G_2$ with $G_i=\Delta
+2\Gamma$, $i=1,2$, i.e. we can't have the $6$ needed rational
components.

If $g-1=21,36$ we use similar arguments with $\sigma_0=26,41$
respectively. In both cases there exist at least 5 rational
components and this is not possible because $a < 5$.

Now consider the case $p\in\Lambda$.

By (\ref{Cbar}) and (\ref{G*}) $b=2a+1$, so $m-1=\bar C^2=4a(a+1)$
and $m=b^2$. Moreover,
$$\bar C.K_{\bar S}=C.K_S+1=b+1$$
Therefore $m=(C.K_S)^2$. Keeping in mind the previous notation, we
get the next possible values
\begin{center}
\begin{tabular}{|c|c|c|c|c|c|c|}\hline
$C.K_S$  & $a$ & $m$ & $g-1$ & $e_f$ & $\sigma_0$ & $l\geq$ \\
\hline 3 & 1 & 9 & 6 & 68 & 12 & 6 \\ \hline 5 & 2 & 25 & 15 & 120
& 20 & 5 \\ \hline 7 & 3 & 49 & 28 & 196 & 33 & 5 \\ \hline
\end{tabular}
\end{center}

Observe that $l$ is, as before, the minimal number of rational
components in $G_0$, so analogous to the case $p\not\in\Lambda$,
all possibilities in the table can't occur because of $a<l$.

{\emph{Case 2}} $\phi$ is 1:1. In this case there exists a divisor
$\bar{C}'$ such that $\phi^*G=\bar{C}+\bar{C}'$.

Denote as before by $F_0$ a singular fiber of $f$, $C_0$ its image
under $\pi$. If $F_0=F_1+...+F_k$ is the decomposition of $F_0$
into irreducible components, we denote by $C_0=C_1+...+C_k$ the
corresponding decomposition for $C_0$ and by $\bar C_0=\bar
C_1+...+\bar C_k$ the corresponding curves and decomposition in
$\bar S$ and by $G_0=G_1+...+G_k$ their images in $\mathbb{F}_2$.

We begin by stating the following:

\begin{lem}\label{lemmaphi} In the previous situation, let $G_1$ be any irreducible
component of $G_0$, then neither $G_1 \sim \Delta$, nor $G_1\sim
\Gamma$ nor $G_1\sim \Delta + 2\Gamma$. In particular, if $G_0\sim
a\Delta + b \Gamma$, then the number of rational irreducible
components of $G_0$ is least or equal than $a$.
\end{lem}

\begin{proof}

Let $G_1$ be equivalent to $\Delta$. Then $\phi_2^*(G_1)\sim 2E$.
Note that there exists a divisor $\bar C_1'$ such that $\bar
C_1+\bar C_1'=\phi_2^*(G_1)$. This implies $\bar C_1=E$, which is
impossible by the definition of $\bar C$.

Now, assume $G_1\equiv \Gamma$, then $\phi^*G_1\equiv K_S-E=C_1+
C'_1$. From this, intersecting with $K_S$, and using that $K_S$ is
nef we obtain that either $C_1.K_S=0$ or $K_S.C'_1=0$. It follows
that there exists a curve $D$ on $S$ with $K_S.D=0$, which
contradicts that ampleness of $\vert K_S \vert$.

Finally, suppose that $G_1=\Delta + 2\Gamma$, in this case $\phi^*
G_1\equiv 2K_S$.

Note that for any decomposition $2K_S= A+A'$ we must have that
both, $A$ and $A'$ must be irreducible and equivalent to $K_S$ and
$A.A'=1$. Indeed, assume $A$ irreducible, that from $2K_S=A+A'$
with easy it follows that $A.K_S=A'.K_S=1$. Therefore, $A^2=1$,
because $2K_S$ is $1-$connected, $A\equiv K_S$ follows from HIT.

One stablished this fact, just note that if $2K_S=C_1+ C'_1$, then
$\phi: C_1\to G_1$ can not be $1:1$, because $C_1$ is a genus $2$
curve.

The last assertion follows from the fact that the only irreducible
rational curves on $\mathbb{F}_2$ are equivalent to either
$\Delta$, $\Gamma$ or $\Delta+ b\Gamma$ with $b\ge 2$.
\end{proof}

Continuing the proof of the Theorem, assume first that
$p\not\in\Lambda$. We have the commutative diagram:
$$\xymatrix{\ar[d]_\Pi\overline{S} \ar[r]^{\phi_2}  & \mathbb{F}_2\ar[d]\\
 S \ar[r]_{\phi_{2K}} & Q},$$
therefore $\phi_{2K}(\Pi(\bar C'))$ contains the singular point of
$Q$ and from this it follows that  $\bar C'=\Pi^*C'$, for some
effective divisor $C'$ in $S$. By (\ref{G*}) $b=2a$ and $\bar C
+\bar C'\sim bK_S$. Therefore,

 $$C.K_S + C'.K_S = (\bar C +\bar C')K_{\bar S} = (bK_S)(K_S + E) = b.$$

The ramifications of $\phi_2$ occurring on $\bar C$ are given by
the intersections of $\bar C$ and $\bar C'$, so we have $\bar
C.\bar C' = \bar C.R = \bar C'.R$. In particular, the right hand
term of the previous equation implies that $C.K_S = C'.K_S$ and $b
= 2C.K_S$.

Moreover,

$$bC.K_S=\phi^*G.\bar C=(\bar C +\bar C').\bar C=m+\bar C.\bar C'=m+5C.K_S.$$

We conclude that  $m=C.K_S(2C.K_S-5)$ and we get the possible
values (keeping in mind the previous notation introduced for
$\sigma_0$ and $l$):

\begin{center}
\begin{tabular}{|l|c|c|c|c|}\hline
$b$  & $C.K_S$ & $m$ & $g-1$  \\
\hline 8 & 4 & 12 & 8  \\ \hline 10 & 5 & 25 & 15
\\ \hline 12 & 6 & 42 & 24  \\ \hline 14 & 7 & 63 &
35  \\ \hline
\end{tabular}
\end{center}
The values $g-1=24$, $35$ are impossible because of the Hodge
Index Theorem.

If $g-1 = 8$ we have that $e_f = 79$ and then there must exists a
singular fiber $F_0$ with at least $14$ singular points. It
follows that $G_0$ has $6$ or more rational components. Using that
$a=4$ and Lemma \ref{lemmaphi}, we get a contradiction. The case
$g-1=15$ follows after similar considerations.

It remains to analyze the case $p\in\Lambda$.

As in the previous case, like $\bar C=C-E$ also $\bar C'=C'-E$.
From this we get $2a-b=-2$. Therefore we have:

$$\phi_2^* G= C+C'-2E= bK_S- 2E.$$
Moreover, $\bar C.R=\bar C'.R$ and then $C.K_S+1=\bar C.K_{\bar
S}=\bar C'.K_{\bar S}$

We get the next formulas

$$\phi_2^*G.K_{\bar S}=b+2=2C.K_S+2$$
and therefore
$$m+5C.K_S=\phi_2^*G.\bar C=(bK_S-2E).(C-E)=bC.K_S-2$$

We conclude that $b=2C.K_S$ and  $m=2(C.K_S)^2-5C.K_S -2$. The
table of possible values is:
\begin{center}
\begin{tabular}{|l|c|c|c|c|}\hline
$C.K_S$  & $b$ & $a$ & $m$ & $g-1$   \\
\hline 4 & 8 & 3 & 10 & 7  \\ \hline 5 & 10 & 4 & 23 &
14 \\ \hline 6 & 12 & 5 & 40 & 23  \\
\hline 7 & 14 & 6 & 61 & 34 \\ \hline
\end{tabular}
\end{center}
If $g-1=23, 34$ we get a contradiction by Hodge Index Theorem. If
$g-1 = 7$, then $e_f = 73$,  therefore there exists a singular
fiber $F_0$ with  at least $13$ singular points and at least $6$
rational components. Taking in consideration that $a=3$ and Lemma
\ref{lemmaphi} we obtain a contradiction. The case $g-1=14$ is
similar.
\end{proof}

\section{The adjoint case}

In this section we consider fibrations $f: X \to \mathbb{P}^1$
satisfying the property that $C$ is an adjoint linear system,
i.e., $C\equiv B+K_S$ with $B$ a big and nef divisor. The typical
example for bearing in mind is $C\equiv nK_S$, i.e. the fibration
$f$ is obtained after blowing up the base locus of a generic
pencil of curves $\Lambda \subset \vert n K_S \vert$.

We collect, for further use, some general elemental facts in the
following:

\begin{lem}\label{AdjointProperties} Assume $C\equiv B+K_S$ with $B$ a big and nef divisor, denoting by
$g_B$ the arithmetic genus of $B$, we have:

\begin{itemize}

\item[i)]
$$ 2(g_B-1)=B^2+B.K_S.$$

\item[ii)]
$$ m=(g_B-1)+(g-1).$$

\item[iii)]
$$(g-1)=(g_B-1)+B.K_S+K_S^2.$$

\item[iv)] $2\leq B.K_S,$ and  if $g\geq 5$, then

$$g+1\leq m.$$
\end{itemize} \end{lem}

\begin{proof} Assertions i)-iii) follow immediately from
adjunction formula. As for iv), note that it is enough to prove
that $2\leq B.K_S,$ because then using i) $g_B-1\geq 2$ and the
desired inequality follows from ii). Now, if $B.K_S=1,$ by Index
Hodge Theorem $B^2=K_S^2=1$ and applying from $m=B^2+2B.K_S
+K_S^2$ we have $m=4.$ On the other hand, from i) and ii)
$m=1+(g-1)=g\geq 5$.
\end{proof}

We start by studding  the case $K_S^2=1$, which  is similar in
nature to Theorems \ref{K^2=2p=3} and \ref{K^2=1p=2}:

\begin{prop}\label{K^2=1p=1} Let $f:X\to \mathbb{P}^1$ be a fibration obtained as
the blowing up of the base locus of a transversal and adjoint
pencil $\Lambda$ on a minimal surface with $K_S^2=1$. Then, $s\ge
7$. \end{prop}

\begin{proof} By Noether Inequality $p_g\le 2$, and by Proposition \ref{CotaGenus} and Theorem \ref{K^2=1p=2}, we can assume $p_g=1$.
It is well known that for a surface with such invariants the
bicanonical map $\phi_{2K_S}$ defines a $4:1$ morphism onto
$\mathbb{P}^2$ (\cite{cat}, \cite{kun}, \cite{tod}):

$$\phi_{2K_S}: S \to \mathbb{P}^2,$$ ramified along a divisor $R\equiv 7K_S$. We consider, as before, the restriction of this map to $C$:

$$
 \xymatrix{
C \ar[r]^{n:1}\ar[d]_{\phi} &\tilde G\ar@{->}[dl]^{j}\\
G &}.
$$

In this case $n$ is a divisor of $4$, $G\subset \mathbb{P}^2$ is
the image of $C$ and $j$ denotes its normalization. Denote by $d$
the degree of $G$.

If we assume $s=6$ we only need, according to Proposition
\ref{CotaGenus} to consider $6\le g \le 16$.

We start by analyzing the case $n=4$: in this case we have
$C=\phi^*G$ and therefore:
$$m=C^2=(\phi^*G)^2=4G^2=4d^2.$$
On the other hand, taking in account that $\phi^*H\equiv 2K_S,$
with $H$ hyperplane section (i.e the divisor associated with
$\mathcal{O}_{\mathbb{P}^2}(1)$) we obtain:
$$C.K_S=\phi^*G.K_S=\phi^*(dH).K_S=2d.$$
Adjunction formula gives $2(g-1)=2d(2d+1).$ The only value  of $g$
that satisfies the relation in the range $6\leq g\leq 16$ is
$g-1=10$ with $d=2$ and $m=16$. Evaluating in Tan's inequality for
$e=4$ (Lemma \ref{TanIneq} ii))we obtain a contradiction.

If $n<4$, then there exists an effective divisor $C'>0$ such that:

$$C+C'\equiv \phi_{2K_S}^* G \equiv 2dK_S.$$

Note that:

$$C'.K_S=2d-C.K_S, $$ and

$$C'^2=4d^2-4dC.K_S+C^2.$$

Moreover, since $h^1(C)=h^2(C)=0$,

$$h^0(C)=\frac{C^2-C.K_S}{2}+2.$$

Next, the only possibility for $h^2(C')=h^0(K_S-C')\ne 0$ is
$C'\equiv K_S$, because $p_g=1$. This would imply $C\equiv K_s$,
which is impossible.

It is easy to prove that $H^0(C)\simeq H^0(C')$ and by
Riemann-Roch we get:

$$h^0(C)=h^0(C')\ge \frac{C'^2-C'.K_S}{2}+2,$$
and substituting the values of $C'^2$ and $C'.K_S$:

$$h^0(C)\ge h^0(C)+ \frac{ 4d^2- (4d-2) C.K_S-2d}{2}.$$
This implies,

$$2d^2-(2d-1)C.K_S -d \le 0,$$
that is equivalent to $d\le C.K_S$.

Now, assume $n=1$. The linear system $\vert 2K_S\vert \mid_C$
defines a base point free linear system on $C$ and the associated
map a $1:1$ cover onto a plane degree $d$ curve. Thus, we have
$d=2C.K_S$ and we obtain a contradiction with the just obtained
bound $d\le C.K_S$.

The case $n=2$ remains to be analyzed: in this case we have
$C.K_S=d.$ Note that the intersections of $C$ and $C'$ gives place
to ramifications points of $\phi_{2K_S}$. Therefore:

$$2dC.K_S-m=C.C'\le C.R=7C.K_S.$$

From this we get:

$$d(2d-7)\le m\le  2d^2-2.$$

It follows that $d\le 5$. In general we have that

\begin{equation} \begin{aligned} e_f & = e(X)+4(g-1) \\ &= 24 -
(1-m) + 4(g-1),\\ & = 23 - d + 6(g-1). \end{aligned}
\end{equation}

Thus, assuming $s=6$ and $d\le 3$, there must exits a singular
fiber $F_0$ of $f$ having at least $(g-1)+ 4$ nodes.  Call
$\sigma_0$ the number of nodes of $F_0$. Note that the number of
nodes of $C_0=\pi(F_0)$ is also $\sigma_0$. Then, $\sigma_0\ge 9$.
On the other hand, the plane curve $G_0= \phi_{2K_S} (C_0)$, being
of degree $d\le 3$ admits at most $3$ nodes. In this way we get
the contradiction $\sigma_0\le 6$.

Similar argumentations lead to contradictions for the cases
$d=4,5$. Indeed, if $d=4$, then $\sigma_0\ge g+3$. We have:

$$g-1=\sum_{i=1}^l (g_i-1) + \sigma_0\ge -l+ g+3,$$
with $g_i$ standing for the geometric genus of the components
$F_i$ of $F_0$. It follows that $l\ge 4$ and therefore $F_0$ (and
in consequence $C_0$) has at least $4$ rational components. From
this it follows that $G_0$ has at least $2$ irreducible rational
components. Taking in account that $G_0$ is a degree $4$ curve we
have that, either $G_0$ contains a line as an irreducible
component or it is the product of two irreducible conics. If $G_0$
is the product of two irreducible conics then it has only $4$
nodes and we get, as before a contradiction, in any other case, if
$L$ is an irreducible component of $G_0$ and $C_0=C_1+...+C_l$,
then
$$\phi_{2K_S}^* L= C_i+C_i',$$ for $C_i$ some rational components
of $C_0$. But then:

$$2K_S \equiv C_i+C_i',$$ and it follows (\cite{bom} Lemma 1, page $181$) that
$C_i=\Delta$, the only effective divisor in $\vert K_S \vert$.
This give a contradiction, since $\Delta$ is a curve of geometric
genus $2$.

Finally the case $d=5$ follows after similar considerations. In
this case $C_0$ admits at least $3$ irreducible rational
components and $G_0$ at least $2$ irreducible rational components.
The only subtle case to be treated careful being the possibility
that $G_0=Q+E$, with $Q$ an irreducible conic and $E$ a singular
irreducible cubic. But in this case either $Q$ or $E$ must
satisfies that it pull back under $\phi_{2K_S}$ is the sum of two
irreducible components $C_i+C_j$ of $C_0$.

Note that $\phi_{2K_S}^*E=C_i+C_j$ is impossible, because then

$$6K_S\equiv C_i+C_j$$ and then $K_S.(C_i+C_j)=6$, that
contradicts $K_S.C_0=5$. On the other hand

$$4K_S=\phi_{2K_S}^*Q=C_i+C_j$$ implies that $C_0$ has exactly $3$
irreducible components: $C_0=C_1+C_2+C_3$ . Suppose $i=1$ and
$j=2$, then $C_3.K_S=1$ and

$$\phi_{2K_S}: C_3 \to E,$$ must be a $2:1$ map onto a degree cubic and we obtain the contradiction

$$2K_S.C_3=6.$$
\end{proof}

Finally we have:

\begin{teo}\label{Adjoint} Let $f:X\to \mathbb{P}^1$ a semi-stable non-isotrivial
fibration obtained as the blow-up of the base locus of an adjoint
pencil $\Lambda$ on the minimal surface $S$. Then if $K_S^2\ge 3$
the number $s$ of singular fibers of $f$ is at least $7$.If
$K_S^2\le 2$ and $\Lambda$ is also transversal, then $s\ge
7$.\end{teo}

\begin{proof}

We assume $s=6$ and $2\le K_S^2$. From Noether's inequality:
$$p_g\le \frac{K_S^2}{2}+2,$$ we have:

$$\chi (\mathcal{O}_S)\le \frac{K_S^2}{2}+3.$$

Applying Lemma \ref{CotaGenus}:

$$K_S^2+ \sqrt{8K_S^2(g-1)+(K_S^2)^2}\le 3K_S^2+18,$$
which implies

$$8K_S^2(g-1)+(K_S^2)^2\le 4(K_S^2)^2+72K_S^2+18^2.$$

Substituting $2(g-1)\ge C.K_S+m= B.K_S+K_S^2+m,$ we get:

$$m\le \frac{72-4B.K_S}{4} + \frac{18^2-(K_S^2)^2}{4K_S^2}.$$

Now, combine the previous bound for $m$ with Lemma \ref{TanIneq}
ii), in order to deduce:

$$19K_S^2+19B.K_S\le 18 + \frac{18^2-(K_S^2)^2}{4K_S^2} +108.$$

Using again the adjoint hypothesis, this amount to:

$$\frac{(K_S^2)}{4} + 19C.K_S\le 166,$$
that is,

$$C.K_S \le 8.$$

Now, use Hodge Index Theorem:

$$m\le \frac{(C.K_S)^2}{K_S^2}\le \frac{64}{K_S^2},$$
and apply one more time Lemma \ref{TanIneq} ii):

$$K_S^2+18 C.K_S\le \frac{64}{K_S^2} +108,$$

$$19K_S^2+18B.K_S \le \frac{64}{K_S^2} +108.$$

Finally, using $2\le B.K_S,$ (Lemma \ref{AdjointProperties} iv))
we arrive to:

$$19K_S^2 \le  \frac{64}{K_S^2} +72.$$

This implies $K_S^2\le 4$. Thus, only the cases $K_S^2=2,3,4$
remains to be discharged. This is easy and  essentially is a
reproduction of the previous argument.

For instance, for case $K_S^2=2$ we have by Lemma \ref{Cotam} and
Proposition \ref{CotaGenus} that $g\le 16$ and $m\le 26$,
moreover, assuming $\Lambda$ is transversal, we can  apply Theorem
\ref{K^2=2p=3} and assume that $\chi (\mathcal{O}_S)\le 3$.
Evaluating Tan's Inequality at $e=5$ (Lemma \ref{TanIneq} iii)) we
get:

$$7K_S^2+6C.K_S\le m +9\chi (\mathcal{O}_S),$$
that, under our fixed values becomes $C.K_S\le 6$. Using Hodge
Index Theorem we obtain $m\le 18$. Evaluating again Tan's
Inequality at $e=4$ we have $C.K_S\le4$ and $m\le 8$. Finally, we
evaluate once again Tan's Inequality at $e=5$ and get the final
contradiction $C.K_S\le 3$ and $g+1\le m\le 4$.

Cases $K_S^2=3,4$ are quite analogous, only that in these cases we
don't need Theorem \ref{K^2=2p=3}, and in consequence the
transversality hypothesis can be avoided.
\end{proof}

\enddocument